\documentclass{article}
\usepackage{cite,url}
\usepackage{amssymb,amsfonts,amsmath,mathrsfs,dsfont}
\usepackage{tikz}
\allowdisplaybreaks[4]
\usepackage{amsthm}
\newtheorem{thm}{Theorem}[section]
\newtheorem{lem}[thm]{Lemma}

\newtheorem{conj}[thm]{Conjecture}

\newtheorem{rem}{Remark}

\newenvironment {Proof of 1}{\noindent {\bf Proof of Theorem \ref{k23}.}}{\hfill\ensuremath{\square}}
\newenvironment {Proof of 2}{\noindent {\bf Proof of Theorem \ref{thm-main}.}}{\hfill\ensuremath{\square}}

\title{
   A note on median eigenvalues of subcubic graphs \thanks{Partly supported by the National Natural Science Foundation of China (Nos.12371354, 11971311, 12161141003) and Science and Technology Commission of Shanghai Municipality (No.~22JC1403602),  National Key R\&D Program of China under Grant No. 2022YFA1006400 and the Fundamental Research Funds for the Central Universities.}
}

\author{Yuzhenni Wang, Xiao-Dong Zhang$^\dag$\\
School of Mathematical Sciences, MOE-LSC, SHL-MAC\\
Shanghai Jiao Tong University, Shanghai 200240, P. R. China
}

\date{}

\begin{document}
\maketitle
\renewcommand{\thefootnote}{}
\footnote{$^\dag$Corresponding author: xiaodong@sjtu.edu.cn.}
	\footnote{E-mail address:
	wangyuzhenni@sjtu.edu.cn(Y. Wang),
	xiaodong@sjtu.edu.cn (X.-D. Zhang)}
\begin{abstract}
 Let $G$ be an simple graph of order $n$ whose adjacency eigenvalues are   $\lambda_1\ge\dots\ge\lambda_n$. The HL--index of $G$ is defined to be
 $R(G)= \max\{|\lambda_{h}|, |\lambda_{l}|\}$ with
  $h=\left\lfloor\frac{n+1}{2}\right\rfloor$ and $ l=\left\lceil\frac{n+1}{2}\right\rceil.$   Mohar  conjectured that  $R(G)\le 1$ for every planar subcubic graph $G$.  In this note, we prove that Mohar's Conjecture holds for every $K_4$-minor-free  subcubic graph. 
  In addition, $R(G)\le 1$ for every  subcubic graph $G$ which contains a subgraph $K_{2,3}$.
\end{abstract}


\section{Introduction}

Let $G=(V(G), E(G))$ be a simple graph  with vertex set $V(G)$ and edge set $E(G)$.  Denote by $|G|$  the order of graph $G$ and $N_G(v)$ the   neighborhood of a vertex $v$ of $G$, and write $N(v)$ when $G$ is clear. For a subset $A\subseteq V(G),$ denote by  $G[A]$  the subgraph of $G$ induced by  $A$, and write $G-A$ for $G[V(G)-A]$ sometimes.

A graph is called  {\it subcubic} if its maximum degree is at most 3. In mathematical chemistry, every subcubic graph is regarded as a chemical graph. It is known that the HOMO--LUMO separation, which is the gap between the Highest Occupied Molecular Orbital(HOMO) and Lowest Unoccupied Molecular Orbital (LUMO),  is related linearly to median eigenvalues of a graph \cite{Fowler2010,Fowler2010-2}. Therefore, it is worthwhile to estimate the median eigenvalues. Fowler and Pisanski \cite{Fowler2010,Fowler2010-2} introduced the notion of HL--index of a graph (see also Jakli\u{c} et al.\cite{Jaklic2012}). For a  simple graph $G$ of order $n$,  let  $\lambda_i(G)$  be  the $i$-th largest eigenvalue of the adjacency matrix of $G$ (counting multiplicities).
 The HL-index of $G$ is defined as
$$R(G)=\max\left\{|\lambda_{h}(G)|, |\lambda_{l}(G)|\right\},$$
 where $h=\left\lfloor\frac{n+1}{2}\right\rfloor$ and $ l=\left\lceil\frac{n+1}{2}\right\rceil.$

In 2015, Mohar \cite{Mohar2015} proved that  $R(G)\le \sqrt{2}$ for every subcubic graph $G$.   Further, he proposed the following conjecture.
\begin{conj}\cite{Mohar2015}\label{con1}
For every subcubic planar graph $G$, $R(G)\le 1$.
\end{conj}

Mohar  \cite{Mohar2013}  confirmed that Conjecture \ref{con1} holds for  bipartite subcubic planar graphs. Later,  Mohar\cite{Mohar2016} proved that $R(G)\le 1$ holds for every bipartite subcubic graph $G$ except the Heawood graph, whose median eigenvalues are $\pm \sqrt{2}$. 
 In addition, Benediktovich\cite{Benediktovich2014} confirmed Conjecture \ref{con1} for subcubic outerplanar graphs.  For more results on median eigenvalues, see  \cite{Li2013,MT2015,Ye2017,Wu2018}.

 A graph $H$ is a {\it minor} (or {\it $H$-minor}) of a graph $G$, if a copy of $H$ can be obtained from $G$ by deleting and/or contracting edges of $G$.
A graph is {\it $\mathcal{H}$-minor-free} if $H$ is not a minor of it for every $H\in \mathcal{H}$. When $\mathcal{H}=\{H\}$, we simply write {\it $H$-minor-free}. In this note, we prove the following result.
\begin{thm}\label{k23}
	For every  subcubic graph $G$ that  contains a subgraph $K_{2,3}$, we have $R(G)\le 1$.
\end{thm}

 In electronic engineering and computer science,  $K_4$-minor-free graphs,  which are also called  series--parallel graphs, are of great interest,  because these graphs  can be applied to model series and parallel electric circuits. More results on series--parallel graphs may be referred to   \cite{Dissaux2023,Eppstein1992} and references therein. 
 Another main result of this note is as follows.
\begin{thm}\label{thm-main}
For every   $K_4$-minor-free subcubic graph $G$, we have  $R(G)\le 1.$
\end{thm}

\begin{rem}
On one hand, Theorem \ref{thm-main} extends the result for subcubic outplanar graphs in \cite{Benediktovich2014}, since  $G$ is outerplanar if and only if neither $K_4$ nor $K_{2,3}$ is a minor of $G$.
 On the other hand, Theorem \ref{thm-main}  confirms that Conjecture~\ref{con1} holds for  $K_4$-minor-free subcubic graphs, as every $K_4$-minor-free graph is planar, equivalently,
 $\{K_5,K_{3,3}\}$-minor-free.
\end{rem}

 The rest of this note is organized as follows. In Section 2, some notations and lemmas are presented. In Section 3, we prove  Theorems \ref{k23} and \ref{thm-main}.
 \section{Preliminaries }
 First, we need the following  two  well-known theorems in spectral graph theory (for example, see  \cite[pp. 17--20]{C2010}).

\begin{thm}[Eigenvalue interlacing theorem]\label{thm2}
		Let $G$ be a simple  graph $G$ of order $n$.
Let $A\subseteq V(G)$ be a vertex set of size $k.$  Then
for  $i\in\{1,2,\cdots,n-k\},$ 
$$\lambda_i(G)\ge \lambda_i(G-A)\ge\lambda_{i+k}(G).
$$

\end{thm}
\begin{thm}\label{thm3}
	Let $G$ be a simple  graph of order $n$. Suppose  $\{E_1,E_2\}$ is a partition of $E(G)$.  For $i\in\{1,2\}$, let $G_i=(V(G),E_i)$ be the spanning graph of $G$.  Then
	\[\lambda
	_i(G)\le \lambda_j(G_1)+\lambda_{i-j+1}(G_2)~(n\ge i\ge j\ge 1),
	\]
	\[\lambda_i(G)\ge \lambda_j(G_1)+\lambda_{i-j+n}(G_2)~(1\le i\le j\le n).
	\]
\end{thm}
We also need a simple fact.
\begin{lem}\label{lem1}
	Let $G$ be a simple graph with two distinct vertices $u$ and $v$ such that $N(u) = N(v)$. Then zero is an adjacency  eigenvalue of $G$. Furthermore, if $G$ is bipartite, then  $R(G)=0$.
\end{lem}
\begin{proof}
	The determinant of the adjacency matrix $A$ of $G$  is equal to 0, so zero is an eigenvalue of $A$.
	Furthermore, if $G$ is bipartite, then the adjacency eigenvalues of $A$ is symmetric with respect to origin 0; so $R(G)=0$.
\end{proof}
Let $G$ be a simple graph of order $n$. A partition $\{A,B\}$ of vertex set of $G$ is called {\it unfriendly} if every vertex in $A$ has at least as many neighbors in $B$ as in $A$, and every vertex in $B$ has at least as many neighbors in $A$ as in $B$. A partition $\{A,B\}$ of vertex set of $G$ is called {\it unbalanced} if $|A|\neq |B|$; {\it balanced}, otherwise.
We include Lemma \ref{thm1} which was initially presented in \cite{Aharoni1990},
 and later applied to median eigenvalues by Mohar in \cite{Mohar2013,Mohar2015}.

\begin{lem}\label{thm1}\cite{Aharoni1990}
	Every graph has an unfriendly partition.
\end{lem}
By applying Theorem  \ref{thm2}, Mohar \cite[Lemma 2.4]{Mohar2013} proved the following lemma.

\begin{lem}\label{lem2} \cite{Mohar2013}
	If $G$ is a subcubic graph with an unbalanced unfriendly partition, then $R(G)\le 1$.
\end{lem}
Finally,
we need the following lemma,  which was verified by Mohar in the proof of  \cite[Lemma 2.1]{Mohar2013}.
\begin{lem}\label{lem-odd}\cite{Mohar2013}
	For every subcubic graph $G$ of odd order, $R(G)\le 1$.
\end{lem}

\section{Proof of Theorems \ref{k23} and \ref{thm-main}.}
In this section, we give the proof of Theorems \ref{k23} and \ref{thm-main}.

\medskip
\begin{Proof of 1}
	Let $G$ be 	a subcubic graph  such that $K_{2,3}$ is a subgraph of $G$,
	where  $V(K_{2,3})=\{x_1, x_2, y_1, y_2, y_3\} \subseteq V(G)$
	with bipartite partition $\{x_1, x_2\}\cup \{y_1, y_2, y_3\}$.
	Take  an unfriendly partition $\{A,B\}$ of $V(G)$, as there exists one
	by Lemma \ref{thm1}.
	If $x_1\in A$ and $x_2\in B$, then
	any way of putting  $y_1,y_2$ and $y_3$ into either $A$ or $B$ can not build the unfriendly partition  $\{A,B\}$.  Thus we may assume that $x_1$ and $x_2$ belong to the same part, say $A$.
	Then $\{y_1,y_2,y_3\}\subseteq B$.
	
	Let $G_1=(V(G),E_1)$  and $G_2=(V(G),E(G)\setminus E_1)$ be two spanning subgraphs of $G$, where $E_1=E(A,B)$ consists of  all  edges with one end in $A$ and the other in $B$.   Then  $G_1$ is a bipartite graph with $N_{G_1}(x_1)=N_{G_1}(x_2)=\{y_1,y_2,y_3\}$.
	 Hence,  $\lambda_h(G_1)=0$ with where $ h=\left\lfloor\frac{n+1}{2}\right\rfloor$ by Lemma \ref{lem1}. In addition, since $\{A,B\}$ is an unfriendly partition of $G$, the graph $G_2$ consists of independent edges and isolated vertices, which implies that $\lambda_1(G_2)\le 1$.
	 Therefore, by Theorem \ref{thm3},
	$$\lambda_h(G)\le \lambda_h(G_1)+\lambda_1(G_2)\le 1,$$
	where $ h=\left\lfloor\frac{n+1}{2}\right\rfloor.$
		By a similar argument, we can also prove that $\lambda_{l}(G)\ge -1$ with  $l=\left\lceil\frac{n+1}{2}\right\rceil$.  Hence $R(G)\le 1$.
\end{Proof of 1}

\medskip
\begin{Proof of 2}
If the assertion of Theorem \ref{thm-main} holds for connected graphs, then it also holds for disconnected graphs by considering all the components. Thus we may assume that $G$ is  a  $K_4$-minor-free subcubic connected graph .
And by Lemma \ref{lem-odd}, we may assume that $|G|=n$ with even $n$.  We use the induction method on even $n$.

If $n=2$, then  $G=K_2$ and the assertion holds. In addition, if $n=4,$ then $G$ must be one of the five graphs: star $K_{1,3}$, 4-cycle $C_4$, kite $K_4-e$ and paw  which is a graph obtained by connecting a vertex  $v$ to one vertex of $K_3$. It is easy to see the second largest eigenvalue of the five graphs at least 1, i.e., $R(G)\le 1$. So the assertion holds.

Suppose the assertion holds for all the graphs of order less than $n$,  where $n\ge 6$. Now we consider the following two cases. 

\textbf{Case 1:}  There is  a cut vertex $v$ in $G$. Then  $G-v$  has a component $G_1$ of odd order. And $G_2:=G-v-V(G_1)$   has  even number of  vertices.  Denote by  $n_1(H)$   the number of eigenvalues of a graph $H$ that are larger than 1.  Since $|G_1|$ is odd,  $R(G_1)\le 1$ by Lemma \ref{lem-odd}, and  $n_1(G_1)\le\frac{|G_1|-1}{2}$.  Since $G$ is subcubic, $G_2$ has at most two component. If $G_2$ is connected or has two components of even order, then apply the inductive hypothesis to components of $G_2$; otherwise, apply Lemma \ref{lem-odd} to two components of odd order of $G_2$.
In all the cases above for $G_2$, we can obtain   $n_1(G_2)\le\frac{|G_2|-2}{2}$. Thus, $$n_1(G-v)=n_1(G_1)+n_1(G_2)\le \frac{|G_1|+|G_2|-3}{2}=\frac{n-4}{2},$$
which implies that   $\lambda_{\frac n2-1}(G-v)\le 1$.  Hence,  by Theorem \ref{thm2}, $\lambda_{\frac n2}(G)\le\lambda_{\frac n2-1}(G-v)\le 1$.
 By a similar argument, we can also prove that $\lambda_{\frac n2+1}(G)\ge -1$. Therefore, $R(G)\le 1$.

\medskip
\textbf{Case 2:} $G$ is a 2-connected graph. If $G$ is a cycle, then the assertion holds.  Thus we may assume that $G$ is not a cycle. Let $C$ be a  longest cycle in $G$.
Throughout this proof, we say a path lies in $C$ (outside of $C$, resp.) if all the edges of the path are in $C$ (are not in $C$, resp.).
Then there is a path $P$ with two end vertices $u, v\in V(C)$  such that  $P$  contains no edges in  $E(C)$.  Let $N(u)\cap V(C)=\{u_1,u_2\}$, and $N(v)\cap V(C)=\{v_1,v_2\}$. Moreover, there is a  $u_1,v_1$--path in $C$ that contains neither $u$ nor $v$.  
If 
$u_1=v_1$ and $u_2=v_2$, then $G=K_{2,3}$, since  $G$ is $K_4$-minor-free and $C$ is a longest cycle of $G$.  It is a contradiction.
Hence we may assume that $u_2\neq v_2$ and have the following  Claim 1.

\medskip
 {\bf Claim 1.} $G-\{u_2,v\}$ is not connected.
\begin{proof}
	We use the method of contradiction.
	Suppose that $G-\{u_2,v\}$ is  connected.
	Then there exists a path $Q$ from  $u$ to $v_2$ in $G-\{u_2,v\}$.
	Let $y$ denote the vertex of  $Q$ that is adjacent to $u$, which is either the vertex of $P$ (only happens when $|V(P)|>2$) or $u_1$.
	Thus there is a $K_4$-minor of $G$ on $\{u,v,v_2,y\}$, a contradiction that  $G$ is $K_4$-minor-free.
\end{proof}
Let $W$ be the set of vertices of the component of $G-\{u_2,v\}$ that contains $u$ (for example, 
 the  black vertices in Figure \ref{fg1}  represent all the vertices in $W$.). 

\begin{figure}[ht]
\centering
    \scalebox{.8}{
  \begin{tikzpicture}
  \draw(-1,1.73)--(-1,-1.73);
  \draw(-1,0.67)--(-0.33,0)--(-1,-0.67);
  \draw(-1.73,1)--(-1.73,-1);
  \draw(0,2)--(1.73,-1);\draw(1,1.73)--(2,0);

  \draw(0,0) circle (2) ;
  \filldraw[fill=white,draw=black](0,2) circle (.085)node at (0,2.3){\large$u_2$} ;
  \filldraw[fill=white,draw=black](2,0) circle (.085) ;
  \filldraw[fill=black,draw=black](-2,0) circle (.06) ;
  \filldraw[fill=white,draw=black](0,-2) circle (.085) node at (0,-2.3){\large$v_2$};
  \filldraw[fill=black,draw=black](-1.73,1) circle (.06) node at (-1.73,1.34){\large$u_1$};
  \filldraw[fill=black,draw=black](-1.73,0) circle (.06);
  \filldraw[fill=black,draw=black](-1,1.73) circle (.06) node at (-1,2.02){\large$u$};
  \filldraw[fill=white,draw=black](1.73,1) circle (.085) ;
  \filldraw[fill=white,draw=black](1,1.73) circle (.085) ;
  \filldraw[fill=black,draw=black](-1.73,-1) circle (.06) node at (-1.73,-1.34){\large$v_1$};
  \filldraw[fill=white,draw=black](-1,-1.73) circle (.085) node at (-1,-2.02){\large$v$};
  \filldraw[fill=white,draw=black](1.73,-1) circle (.085) ;
  \filldraw[fill=white,draw=black](1,-1.73) circle (.085) ;
  \filldraw[fill=black,draw=black](-1,0.67) circle (.06) ;
  \filldraw[fill=black,draw=black](-1,-0.67) circle (.06) ;
  \filldraw[fill=black,draw=black](-0.33,0) circle (.06) ;
  \filldraw[fill=white,draw=black](0.57,1) circle (.085) ;
  \filldraw[fill=white,draw=black](1.15,0) circle (.085) ;
  \end{tikzpicture}
  }
  \caption{An example for $G$}\label{fg1}
  \end{figure}
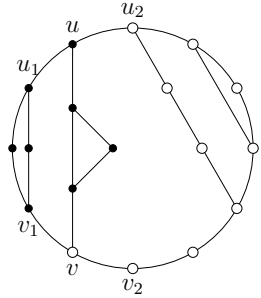

\medskip
{\bf Claim 2.}
If $u$  is adjacent to  $v$, then
$G[W-u]$ is connected and  a  component of $G-\{u,v\}$; if  $u$ is not adjacent to  $v$, then  there are  two connected components in
$G[W-u]$  which  are also two components of $G-\{u,v\}$.

\begin{proof}
	If  $u$ is adjacent to  $v$, then  $u$ has degree 1 in $G[W]$. Thus  $G[W-u]$ is still  connected and a component of $G-\{u,u_2,v\}$.
	
If  $u$ is not adjacent to  $v$, then  $|V(P)|>2$.
	Then $G[W-u]$ has two components. In fact, if   $G[W-u]$ is
	connected, then   there is a path from $u_1$ to $y$ in $G[W-u]$  with  $y\in N(u)\cap V(P)$. Hence there is a $K_4$-minor in $G$ on $\{u,v,u_1,y\}$,  which contradicts to that $G$ is $K_4$-minor-free. Therefore $G[W-u]$ is disconnected. Moreover, since the degree of $u$ in $G[W]$ is 2,  there are  exactly two  components in $G[W-u]$.
	
	Let $H_1$ be  the
	 component of $G[W-u]$ containing $u_1$ and  $H_2$ be the rest component of $G[W-u]$.
	    Then by the definition of $W$, $H_1$ and $H_2$ are also two components of $G-\{u,u_2,v\}$.
	Since $G$ is $K_4$-minor-free, there is no paths in    $G-\{u,v\}$ from $u_2$ to vertices of $H_1$ or $H_2$. Hence $H_1$ and $H_2$ are two components of $G-\{u,v\}$.
\end{proof}
If $|W|$ is even,  let $G_1:=G[W]$ and  $x:=u_2$; if $|W|$ is odd, let $G_1:=G[W-u]$ and  $x:=u$. Then $|G_1|$ is  even. Furthermore, by the  definition of $W$ and  Claim 2,  the components of $G_1$ are always the  components  of  $G-\{x,v\}$. Hence by the inductive hypothesis on $G_1$, we have $n_1(G_1)\le\frac{|G_1|-2}{2}$. In addition,
let  $G_2:=G-\{x,v\}-V(G_1)$. Then $G_2$ also has even number of vertices.
Since $u_2\neq v_2$, 
we have $v_2\in V(G_2)$; so $|G_2|\ge 2$.
Since  $G$ is 2-connected,  there are  at most two components in $G_2$. 
By Lemma \ref{lem-odd} and the inductive hypothesis, 
we can obtain $n_1(G_2)\le \frac{|G_2|-2}{2}$. 
Therefore, we have
 $$n_1(G-\{u_2,x\})\le n_1(G_1)+n_1(G_2)\le\frac{|G_1|+|G_2|-4}{2}=\frac{n-6}{2},$$ which implies that $\lambda_{\frac n2-2}(G-\{u_2,x\})\le 1$. By Theorem \ref{thm2}, $\lambda_{\frac n2}(G)\le\lambda_{\frac n2-2}(G-\{u_2,x\})\le 1$.

By a similar argument, we can also prove that $\lambda_{\frac n2+1}(G)\ge -1$. Therefore, $R(G)\le 1$, and we complete the proof.	
\end{Proof of 2}

\section*{Data Availability} No data was used for the research described in the article.

\subsection*{Acknowledgements}
The authors would be grateful to the referees for their valuable suggestions and comments.

\end{document}